\documentclass[11pt]{article}
\usepackage[english]{babel}
\usepackage[latin1]{inputenc}
\usepackage{amsmath}
\usepackage{amsfonts}
\usepackage{amssymb}
\usepackage{amsthm}
\usepackage{mathrsfs}

\newtheorem{thm}{Theorem}

\newtheorem{rem}[thm]{Remark}
\newtheorem{cor}[thm]{Corollary}
\newtheorem{prop}[thm]{Proposition}

\begin{document}

\title{{\bf Green functions and Martin compactification \\ for killed random walks 
       related to SU(3)}}

\author{{\large Kilian Raschel\footnote{Laboratoire de Probabilit\'es et 
         Mod\`eles Al\'eatoires, Universit\'e Pierre et Marie Curie,
         4 Place Jussieu 75252 Paris Cedex 05, France. 
         E-mail : \texttt{kilian.raschel@upmc.fr}}}}

\date{May 18, 2010}

\maketitle

\vspace{-4mm}

\begin{abstract}
     We consider the random walks killed at the boundary of the quarter plane, with 
     homogeneous non-zero jump probabilities to the eight nearest neighbors and 
     drift zero in the interior, and which admit a positive harmonic polynomial 
     of degree three. For these processes, we find the asymptotic of the Green 
     functions along all infinite paths of states, and from this we deduce that 
     the Martin compactification is the one-point compactification.
\end{abstract}

{\raggedright \textit{Keywords\hspace{-0.2mm} :\hspace{-0.2mm} killed random walks, Green functions, Martin compactification, uniformization.}}
\vspace{3mm}
{\raggedright \textit{AMS $2000$ Subject Classification\hspace{-0.2mm} :\hspace{-0.2mm} primary 60G50, 31C35 ; secondary 30F10.}}

\vspace{-5mm}

\section{Introduction and main results}
\label{Intro}

First introduced for Brownian motion by R.\ Martin in 1941, the concept of Martin compactification
has then been extended for countable discrete time Markov chains by J.\ Doob and G.\ Hunt at the 
end of the fifties. The purpose of this theory is to describe the asymptotic behavior of the Markov 
chains and also to characterize all their non-negative superharmonic and harmonic functions, see 
\textit{e.g.}~\cite{Dynkin}.

For a transient Markov chain with state space $E$, the \emph{Martin compactification} of $E$ is the 
smallest compactification $\hat{E}$ of $E$ for which the Martin kernels $z\mapsto k_{z}^{z_{0}}=
G_{z}^{z_{0}}/G_{z}^{z_{1}}$ extend continuously -- by $G_{z}^{z_{0}}$ we mean the \emph{Green functions} 
of the process, \textit{i.e.}\ the mean number of visits made by the process at $z$ starting at $z_{0}$, 
and we note $z_{1}$ a reference state. $\hat{E}\setminus E$ is usually called the \emph{full Martin boundary}. 
Clearly, for $\alpha\in\hat{E}$, $z_{0}\mapsto k_{\alpha}^{z_{0}}$ 
is superharmonic~; then $\partial_{m}E=\{\alpha \in \hat{E}\setminus E : 
z_{0}\mapsto k_{\alpha}^{z_{0}}\text{ is minimal harmonic}\}$ is called the 
\emph{minimal Martin boundary} -- a harmonic function $h$ is said minimal if $0\leq \tilde{h}\leq h$ 
with $\tilde{h}$ harmonic implies $\tilde{h}=c h$ for some constant $c$. Then, every superharmonic 
(resp.\ harmonic) function $h$ can be written as $h(z_{0})=\int_{\hat{E}}k_{z}^{z_{0}}\mu(\text{d}z)$
(resp.\ $h(z_{0})=\int_{\partial_{m}E}k_{z}^{z_{0}}\mu(\text{d}z)$), where $\mu$ is some finite 
measure, uniquely characterized in the second case above. 

\medskip

The case of the \emph{homogeneous} random walks in $\mathbb{Z}^{d}$ is now completely understood. 
First, their minimal Martin boundary is found in~\cite{DOO1}, thanks to Choquet-Deny theory. 
Furthermore, in the case of a \emph{non-zero drift}, P.~Ney and F.~Spitzer find, in their 
well-known paper \cite{NS}, the asymptotic of the Green functions, by using exponential changes 
of measure and the local limit theorem~; this gives consequently the concrete realization of 
the Martin compactification, in that case the sphere. Additionally, in the case of a \emph{drift 
zero}, the asymptotic of the Green functions is computed in~\cite{SPI}~; it follows that the 
Martin compactification consists in the one-point compactification.

\medskip

Results on Martin boundary for \emph{non-homogeneous} random walks are scarcer and more recent. 
We concentrate here our analysis on important and recently extensively studied examples
that are the random walks in $\mathbb{Z}^{d}$ killed at the boundary of cones. They are related to
many areas of probability, as \textit{e.g.}\ to non-colliding random walks or quantum processes.

\medskip
    
On the one hand, the case of the \emph{non-zero drift} is now rather well studied.

In~\cite{Bi3}, P.\ Biane considers quantum random walks on the dual of 
compact Lie groups and, by restriction, arrives at classical random walks 
with non-zero drift killed at the boundary of the Weyl chamber of the dual. 
Solving an equation of Choquet-Deny type, he finds the minimal Martin boundary 
of these processes.

When the compact Lie group is SU($d$) and the associated random walk has non-zero drift,
the Martin compactification is obtained in~\cite{COL}, by finding the asymptotic 
of the Green kernels.

Recently, in~\cite{III}, I.\ Ignatiouk-Robert obtains the Martin compactification of 
the random walks in $\mathbb{Z}_{+}^{d}$ with non-zero drift and killed at the boundary. 
She uses there an original approach based on large deviations theory in order to compute 
the asymptotic of the Martin kernels. Unfortunately, her methods seem quite difficult to 
extend up to the case of the drift zero. Also, they do not provide the asymptotic of the 
Green functions.
   
This asymptotic in the case of the dimension $d=2$ is found in~\cite{KR}.

\medskip

On the other hand, results on Martin boundary for killed random walks 
with \emph{drift zero} are quite rare. The simplest example of the 
cartesian product is due to~\cite{PW}. A more interesting case comes
again from quantum processes~: in~\cite{Bi5}, P.~Biane shows that the 
minimal Martin boundary of the random walk with zero drift and killed 
at the boundary of the Weyl chamber of the dual of SU($d$) is reduced 
to one point.
     
\medskip

It is immediate from~\cite{Bi1} that this classical random walk in the 
Weyl chamber of the dual of SU($d$) is, for $d=3$, the random walk spatially 
homogeneous on the lattice $\{i+j\exp(\imath \pi/3),\ (i,j)\in \mathbb{Z}^{2}\}$ 
with jump probabilities as represented on the left of Picture~\ref{Lie}. 

 \begin{figure}[!ht]
 \begin{center}
 \begin{picture}(310.00,100.00)
 \includegraphics{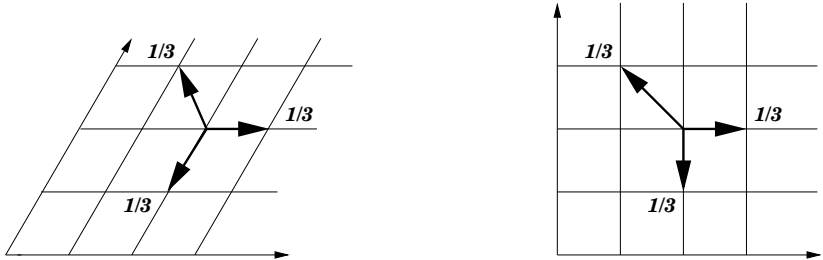}
 \end{picture}
 \end{center}
 \caption{Random walk in the Weyl chamber of the dual of SU(3)}
 \label{Lie}
 \end{figure}

Obviously, a suitable linear transformation maps the lattice $\{i+j\exp(\imath \pi/3)
: (i,j)\in \mathbb{Z}^{2}\}$ into $\mathbb{Z}^{2}$, see Picture~\ref{Lie}~; in this 
way, the Weyl chamber $\{i+j\exp(\imath \pi/3) : (i,j)\in \mathbb{Z}_{+}^{2}\}$
becomes $\mathbb{Z}_{+}^{2}$. For $d=3$, the killed random walk considered by
P.~Biane in~\cite{Bi1} can therefore be viewed as an element of 

\medskip

{\raggedright $\mathscr{P}=\{$random walks in $\mathbb{Z}_{+}^{2}$ 
with non-zero jump probabilities $(p_{i,j})_{-1\leq i,j\leq 1}$}

{\hfill to the eight nearest neighbors, with drift zero and killed at the boundary$\}$}

\medskip
 
{\raggedright with jump 
probabilities as represented on the right of Picture~\ref{Lie} -- above,
``drift zero'' means that $p_{1,1}+p_{1,0}+p_{1,-1}=p_{-1,1}+p_{-1,0}+
p_{-1,-1}$ and $p_{1,1}+p_{0,1}+p_{-1,1}=p_{1,-1}+p_{0,-1}+p_{-1,-1}$. 
In this setting, P.~Biane proves, in~\cite{Bi5}, that $(i_{0},j_{0})\mapsto i_{0}j_{0}(i_{0}+
j_{0})$ is the only positive harmonic function for this process.}

By the same methods, it can certainly be shown that there is only one positive 
harmonic function for the ``dual'' walk, namely for the random walk with jump 
probabilities $p_{-1,0}=p_{0,1}=p_{1,-1}=1/3$. In particular, if we set 
$\mathscr{P}^{c}=\{$random walks of $\mathscr{P}$ such that $p_{0,-1}=
p_{-1,1}=p_{1,0}=\mu$, $p_{-1,0}=p_{0,1}=p_{1,-1}=\nu$, $\mu+\nu=1/3\}$ 
--~in other words, $\mathscr{P}^{c}$ is the set of all cartesian products of the random walk
on the dual of SU(3) with its dual, see on the left of Picture~\ref{Part_Fami} below~--, it follows
from~\cite{PW} that any process of $\mathscr{P}^{c}$ has also a minimal Martin boundary 
reduced to one point.

In this paper, we introduce the set 
     \begin{equation*}
          \mathscr{P}_{1,0}=\{\text{random walks of}\ \mathscr{P} 
          \ \text{for which}\ (i_{0},j_{0})\mapsto i_{0}j_{0}(i_{0}+j_{0})\ \text{is harmonic}\}. 
     \end{equation*}
Note that we have $\mathscr{P}^{c} \subset \mathscr{P}_{1,0}$, but we will see, 
in Remark~\ref{description_walks}, that the inclusion is strict. More generally, 
we define  
     \begin{equation}
     \label{description_jump_probabilities}
          \mathscr{P}_{\alpha,\beta}=\{\text{random walks of}\ \mathscr{P}\ \text{for 
          which}\ (i_{0},j_{0})\mapsto i_{0}j_{0}(i_{0}+\alpha j_{0}+\beta)\ \text{is harmonic}\}.
     \end{equation}

Since any harmonic function for a killed process takes the value zero on the boundary, 
$\mathscr{P}_{\alpha,\beta}$ is in fact exactly the set of \emph{all} killed random walks 
in $\mathbb{Z}_{+}^{2}$ to the eight nearest neighbors admitting a harmonic polynomial 
of degree three. 

The description of the set $\mathscr{P}_{\alpha,\beta}$ in terms of the $(p_{i,j})_{i,j}$ 
is rather cumbersome but not difficult to obtain, it is postponed until Remark
\ref{description_walks}. Let us just note here that if $\alpha>2$ or $\alpha<1/2$, 
then for all $\beta$, $\mathscr{P}_{\alpha,\beta }=\emptyset$~; if $\alpha=1/2$ or $\alpha=2$, 
then for all $\beta\neq 0$, $\mathscr{P}_{\alpha,\beta }=\emptyset$, and $\mathscr{P}_{\alpha,0}$ 
is reduced to one walk~; and if $\alpha \in]1/2,2[$ and $|\beta|$ is small enough, then 
$\mathscr{P}_{\alpha,\beta }$ is a (non-empty) set with two free parameters, properly 
described in Remark~\ref{description_walks}. We have represented on the right of Picture~\ref{Part_Fami} 
an example of a process belonging to $\mathscr{P}_{\alpha,0}$, for any $\alpha\in[1/2,2]$.

\medskip

Moreover, note that considering in this paper $\mathscr{P}_{\alpha,\beta}$ is all the more 
natural as the set $\{$random walks of $\mathscr{P}$ for which $(i_{0},j_{0})\mapsto i_{0}
j_{0}$ is harmonic$\}$ is studied in \cite{Ras}.

 \begin{figure}[!ht]
 \begin{center}
 \begin{picture}(305.00,100.00)
 \includegraphics{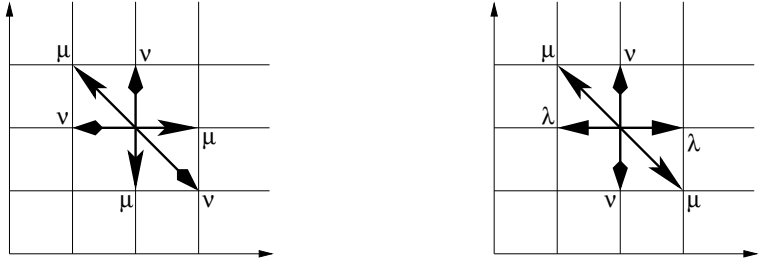}
 \end{picture}
 \end{center}
 \caption{On the left, a generic walk of $\mathscr{P}^{c}$ ($\mu+\nu=1/3$)~; on the 
          right, an~example of walk of $\mathscr{P}_{\alpha,0}$ ($\lambda=\alpha(\alpha-1/2)/
          [2-\alpha+2\alpha^{2}]$, $\mu=(\alpha/2)/[2-\alpha+2\alpha^{2}]$ and~$\nu=
          (1-\alpha/2)/[2-\alpha+2\alpha^{2}]$)}
 \label{Part_Fami}
 \end{figure}

Our first result deals with the Green functions 
--~below, $(X,Y)$ denotes the coordinates of the random walk 
and $\mathbb{E}_{(i_{0},j_{0})}$ the conditional expectation 
given $(X(0),Y(0))=(i_{0},j_{0})$~--
     \begin{equation}
     \label{def_Green_functions}
          G_{i,j}^{i_{0},j_{0}}=\mathbb{E}_{(i_{0} , j_{0})}
          \Bigg[ \sum_{k\geq 0}\textbf{1}_{\left\{
          \left(X\left(k\right),Y\left(k\right)\right)=
          \left(i,j \right) \right\}}\Bigg],
     \end{equation} 
and, more precisely, with their asymptotic along \emph{all} paths of states.
     \begin{thm}
     \label{Main_theorem_Green_functions}
     Suppose that the process belongs to $\mathscr{P}_{\alpha,\beta}$.
     Then the Green functions~(\ref{def_Green_functions}) admit the following
     asymptotic as $i+j\to \infty$ and $j/i\to \tan(\gamma)$, $\gamma$ 
     lying in $[0,\pi/2]$~:
     \begin{equation}
     \label{main_theorem_Green_functions}
          G_{i,j}^{i_{0},j_{0}} \sim C
          i_{0}j_{0}\left(i_{0}+\alpha j_{0}+\beta\right) 
          \frac{ i j (i+\alpha j)}{\big[i^{2}+\alpha i j 
          +\alpha^{2}j^{2}\big]^{3}},
     \end{equation}
     where $C>0$ depends only on the parameters $(p_{i,j})_{i,j}$ and is made explicit in the proof.
\end{thm}
In the particular case of the random walk killed at the boundary of the Weyl chamber
of the dual of SU($3$), the asymptotic~(\ref{main_theorem_Green_functions}) is, for 
$\gamma\in]0,\pi/2[$, proved in \cite{Bi1}. Theorem~\ref{Main_theorem_Green_functions} 
actually completes this result for that very particular random walk and, in fact, gives 
the asymptotic of the Green functions for a much larger class of processes. 

In addition, Theorem~\ref{Main_theorem_Green_functions} has the following consequence.  
    
\begin{cor}
\label{corollary_Martin_boundary}
The Martin compactification of any process belonging to $\mathscr{P}_{\alpha,\beta}$ 
is the one-point compactification.
\end{cor}

Furthermore, the asymptotic~(\ref{main_theorem_Green_functions}) of the 
Green functions in the two limit cases $\gamma=0$ and $\gamma=\pi/2$ 
enables us to obtain the asymptotic of the absorption probabilities
     \begin{equation}
     \label{absorption_probabilities}
          \left.\begin{array}{ccc}
          \displaystyle \phantom{\frac{n}{n}}h_{i}^{i_{0},j_{0}}&=&\displaystyle \mathbb{P}_{\left(i_{0} , j_{0}\right)}
          \left[\exists k\geq 1 : (X(k),Y(k))= (i,0)\right],\phantom{\frac{n}{n}}\\
          \displaystyle \phantom{\frac{n}{n}}\widetilde{h}_{j}^{i_{0},j_{0}}&=&\displaystyle \mathbb{P}_{\left(i_{0},j_{0}
          \right)}\left[\exists k\geq 1 : (X(k),Y(k))= (0,j)\right].\phantom{\frac{n}{n}}
          \end{array}\right.
     \end{equation}
Indeed, the absorption probabilities~(\ref{absorption_probabilities}) are related to the 
Green functions~(\ref{def_Green_functions}) through 
     \begin{eqnarray*}
          h_{i}^{i_{0},j_{0}}&=&p_{1,-1}G_{i-1,1}^{i_{0},j_{0}}+
                                p_{0,-1}G_{i,1}^{i_{0},j_{0}}+
                                p_{-1,-1}G_{i+1,1}^{i_{0},j_{0}},\\
          \widetilde{h}_{j}^{i_{0},j_{0}}&=&
          p_{-1,1}G_{1,j-1}^{i_{0},j_{0}}+p_{-1,0}G_{1,j}^{i_{0},j_{0}}+
          p_{-1,-1}G_{1,j+1}^{i_{0},j_{0}}, 
     \end{eqnarray*}
so that, from Theorem~\ref{Main_theorem_Green_functions},
we immediately obtain the following result.
\begin{cor}
     Suppose that the process belongs to $\mathscr{P}_{\alpha,\beta}$.
     Then the absorption probabilities (\ref{absorption_probabilities})
     admit the following asymptotic as $i\to \infty$~:
     \begin{equation*}
     \label{main_theorem_absorption_probabilities}
          h_{i}^{i_{0},j_{0}} \sim C(p_{1,-1}+p_{0,-1}+p_{-1,-1})
          i_{0}j_{0}\left(i_{0}+\alpha j_{0}+\beta\right)\frac{1}{i^{4}},
     \end{equation*}
     where $C>0$ is the same constant as is the statement of
     Theorem~\ref{Main_theorem_Green_functions}. 
     
     The same asymptotic holds for $\tilde{h}_{i}^{i_{0},j_{0}}$, after having replaced
     $(p_{1,-1}+p_{0,-1}+p_{-1,-1})$~above by $(p_{-1,1}+p_{-1,0}+p_{-1,-1})/\alpha^{5}$.
\end{cor}
The asymptotic of the absorption probabilities in the case of a non-zero
drift being obtained in~\cite{KR}, Corollary~\ref{main_theorem_absorption_probabilities} 
thus gives an example of the behavior of these probabilities in the case of a drift zero.

\medskip

In order to prove Theorem~\ref{Main_theorem_Green_functions}, we are going to develop
methods initiated in~\cite{FIM} and based on complex analysis, what will allow us to 
express \emph{explicitly} the Green functions~(\ref{def_Green_functions}). Indeed, 
in~\cite{FIM}, the authors G.~Fayolle, R.~Iasnogorodski and V.~Malyshev elaborate a 
profound and ingenious analytic approach for studying the 
stationary probabilities for random walks to the eight nearest neighbors in the quarter 
plane supposed ergodic,~\textit{i.e.}\ such that $p_{1,1}+p_{1,0}+p_{1,-1}<p_{-1,1}+
p_{-1,0}+p_{-1,-1}$ and $p_{1,1}+p_{0,1}+p_{-1,1}<p_{1,-1}+p_{0,-1}+p_{-1,-1}$.

We are going to see here that this analytical approach can be extended up to the case 
of the random walks in the quarter plane with drift zero and killed  at the boundary~: 
Section~\ref{h_x_z} of this paper first broadens the analysis begun in Part~6 of~\cite{FIM} 
for the drift zero, and then shows how this applies in the case of the random walks of 
$\mathscr{P}_{\alpha,\beta}$.

It is worth noting that this approach \textit{via} complex analysis is intrinsic 
to the dimension $d=2$~; for this reason, it seems really a difficult task to 
generalize it in higher~dimension.  

\medskip

Let us conclude this introductory part by describing the set $\mathscr{P}_{\alpha,\beta}$ 
defined in~(\ref{description_jump_probabilities}) in terms of the jump probabilities
$(p_{i,j})_{i,j}$.

\begin{rem}
\label{description_walks}
The fact that the two drifts are equal to zero gives two equations and the fact that
the sum of the jump probabilities is one yields an other one. Moreover, the harmonicity 
of $h(i_{0},j_{0})=i_{0}j_{0}(i_{0}+\alpha j_{0}+\beta)$, which reads $h(i_{0},j_{0})=
\sum_{i,j}p_{i,j}h(i_{0}+i,j_{0}+j)$, leads to ten equations, by identification of the 
coefficients of the third-degree polynomials above. It turns out that some of these 
equations are trivial and that some other ones are linearly dependent, we finally obtain 
six equations linearly independent. We can therefore express all the eight jump probabilities 
$(p_{i,j})_{i,j}$ in terms of $p_{1,1}$ and $p_{1,0}$ only, and we obtain~:

\vspace{1mm}

{\raggedright $*\ p_{-1,0}=-[\alpha(1-2\alpha-\beta)+8p_{1,1}+(4-3\alpha+
2\alpha^{2}+\alpha\beta)p_{1,0}]/[\alpha(1+2\alpha+\beta)]$,}

{\raggedright $*\ p_{-1,1}=[\alpha(1-\alpha-\beta)+2(4+3\alpha+2\alpha^{2}+\alpha\beta)
p_{1,1}+2(
2+\alpha^{2}+\alpha\beta)p_{1,0}]/[2\alpha(1+2\alpha+\beta)]$,}

{\raggedright $*\ p_{0,1}=-[-(1+\alpha+\beta)+
4(2+2\alpha+\beta)p_{1,1}+2(2+\alpha+\beta)p_{1,0}]/[2(1+2\alpha+\beta)]$,}

{\raggedright $*\ p_{1,-1}=[\alpha^{2}+(-1+
2\alpha-\beta)p_{1,1}-(1+\beta+2\alpha^{2})p_{1,0}]/[1+2\alpha+\beta]$,}

{\raggedright $*\ p_{0,-1}=-[(-1-3\alpha-\beta+4\alpha^{2})+4(-2+2\alpha-\beta)p_{1,1}+(-4+6
\alpha-2\beta-8\alpha^{2})p_{1,0}]/[2(1+2\alpha+\beta)]$,}

{\raggedright $*\ p_{-1,-1}=[\alpha(1-3\alpha-\beta+2\alpha^{2})+2(4-3\alpha+2\alpha^{2}-\alpha\beta)p_{1,1}+2(2-3\alpha
+3\alpha^{2}-2\alpha^{3})p_{1,0}
]/[2\alpha(1+2\alpha+\beta)]$.}
\end{rem}
The properties of $\mathscr{P}_{\alpha,\beta}$ mentioned below
(\ref{description_jump_probabilities}) are immediately obtained
by studying the sign of the jump probabilities above in 
terms of $\alpha$, $\beta$, $p_{1,1}$ and $p_{1,0}$.

\section{Explicit expression of the Green functions}
\label{h_x_z}

Section~\ref{h_x_z} aims at obtaining an explicit expression
of the Green functions~(\ref{def_Green_functions}) --~what we will
succeed in doing in Theorem~\ref{exp_Green_functions} below. This
forthcoming expression of the Green functions will be, in turn, the 
starting point of Section~\ref{Martin_boundary}, where we will find 
their asymptotic.

In order to prove Theorem~\ref{exp_Green_functions}, we need to state two
results, namely Equation
(\ref{functional_equation}) and  Proposition \ref{continuation_h_h_tilde_covering}~:
Equation (\ref{functional_equation}) is a functional equation between the generating 
function of the Green functions~(\ref{def_Green_functions}) and the ones of the 
absorption probabilities~(\ref{absorption_probabilities}), and Proposition
\ref{continuation_h_h_tilde_covering} establishes some quite important 
properties of the generating functions of the absorption probabilities.

The proof of Proposition~\ref{continuation_h_h_tilde_covering} turns out to require 
considering the Riemann surface defined by $\{(x,y)\in (\mathbb{C}\cup \{\infty\})^{2} 
: \sum_{i,j}p_{i,j}x^{i}y^{j}=1\}$, for this reason we begin Section~\ref{h_x_z} by 
studying --~and, in fact, by uniformizing~-- this surface. 

It seems of interest to us to introduce this Riemann surface in whole generality~;
this is why, at the beginning of Section~\ref{h_x_z}, we are going to 
suppose that the process belongs to $\mathscr{P}$ -- and not necessarily 
to $\mathscr{P}_{\alpha,\beta}$.

\medskip

To begin with, we define the generating functions of the Green functions
(\ref{def_Green_functions}) and of the absorption probabilities
(\ref{absorption_probabilities}) by
     \begin{equation}
     \label{def_generating_functions}
          G^{i_{0},j_{0}}(x,y) = \displaystyle \sum_{i,j\geq 1} G_{i,j}^{i_{0},j_{0}}x^{i-1} y^{j-1}, 
          \ \ h^{i_{0},j_{0}}\left(x\right)=\sum_{i\geq 1}h_{i}^{i_{0},j_{0}}x^{i},\
          \ \ \widetilde{h}^{i_{0},j_{0}}\left(y\right)=\sum_{j\geq 1}\widetilde{h}_{j}^{i_{0},j_{0}}y^{j}
     \end{equation}
and $h_{0,0}^{i_{0},j_{0}}=\mathbb{P}_{(i_{0} , j_{0})}[\exists k\geq 1 : (X(k),Y(k))=  (0,0)]$. 
Of course, $G^{i_{0},j_{0}}$, $h^{i_{0},j_{0}}$ and $\tilde{h}^{i_{0},j_{0}}$ are holomorphic 
in their unit disc. With these notations, we can state the following functional equation on $\{(x,y)\in 
\mathbb{C}^{2} : |x|<1, |y|<1 \}$~:
     \begin{equation}
     \label{functional_equation}
          Q\left(x,y\right)G^{i_{0},j_{0}}\left(x,y\right) =
          h^{i_{0},j_{0}}\left(x\right)+\widetilde{h}^{i_{0},j_{0}}
          \left( y \right)+h_{0,0}^{i_{0},j_{0}} -x^{i_{0}} y^{j_{0}},
     \end{equation}
where $Q(x,y)= x y \big[ \sum_{i,j}p_{i,j}x^{i}y^{j}  -1\big]$.
Equation~(\ref{functional_equation}) 
is obtained exactly as in Subsection~2.1 of \cite{KR}.


The polynomial $Q(x,y)$ defined above can obviously be written as 
     \begin{equation*}
          Q(x,y) = a(x) y^{2}+ b(x) y+ c(x) = 
          \widetilde{a}(y) x^{2}+\widetilde{b}(y) 
          x + \widetilde{c}(y), 
     \end{equation*}
with        
     \begin{equation*}
          a (x)  =p_{1,1}x^{2}+p_{0,1}x+p_{-1,1},\ b (x)  = p_{1,0}x^{2}-x+
          p_{-1,0},\ c (x)  = p_{1,-1}x^{2}+p_{0,-1}x+p_{-1,-1},
     \end{equation*}     
     \begin{equation*}
          \widetilde{a} (y)  =p_{1,1}y^{2}+p_{1,0}x+p_{1,-1},
          \ \widetilde{b} (y)  = p_{0,1}y^{2}-y+
          p_{0,-1},\ \widetilde{c} (y)  = p_{-1,1}y^{2}+p_{-1,0}y+p_{-1,-1}.
     \end{equation*}       

Let us also define the polynomials 
     \begin{equation*}
          d(x)=b(x)^{2}-4a(x)c(x),
          \ \ \ \ \ \widetilde{d}(y)=
          \widetilde{b}(y)^{2}-4\widetilde{a}(y)
          \widetilde{c}(y). 
     \end{equation*}

It is proved in Part~2.3 of~\cite{FIM} that for any random walk of
$\mathscr{P}$, $d$ (resp.\ $\tilde{d}$) has one simple root in $]-1,1[$, 
that we call $x_{1}$ (resp.\ $y_{1}$), a double root at $1$, and a simple 
root in $\mathbb{R}\cup\{\infty\}\setminus [-1,1]$, that we note $x_{4}$  
(resp.\ $y_{4}$).

For example, in the case of ${\rm SU}(3)$, \textit{i.e.}\ for the random walk 
with transition probabilities as in Picture~\ref{Lie}, we immediately obtain 
$x_{1}=0$, $y_{1}=1/4$, $x_{4}=4$ and $y_{4}=\infty$.

From a general point of view, it is shown in Part~2.3 of~\cite{FIM} that $x_{1}$ (resp.~$y_{1}$) 
is positive, zero or negative depending on whether ${p_{-1,0}}^{2}-4p_{-1,1}p_{-1,-1}$ 
(resp.\ ${p_{0,-1}}^{2}-4p_{1,-1}p_{-1,-1}$) is positive, zero or negative, and that 
$x_{4}$ (resp.\ $y_{4}$) is positive, infinite or negative depending on whether 
${p_{1,0}}^{2}-4p_{1,1}p_{1,-1}$ (resp.\ ${p_{0,1}}^{2}-4p_{1,1}p_{-1,1}$) is positive, 
zero or negative.

\medskip

Let us now have a look to the surface defined by $\{(x,y)\in (\mathbb{C}\cup 
\{\infty\})^{2} : Q(x,y)=0\}$, that we note $\mathscr{Q}$ for the sake of 
briefness. Note first that $Q(x,y)=0$ is equivalent to $[b(x)+2a(x)y]^{2}=d(x)$ 
or to $[\tilde{b}(y)+2\tilde{a}(y)x]^{2}=\tilde{d}(y)$. As a consequence, it follows 
from the particular form of $d$ or of $\tilde{d}$ (two distinct simple roots 
different from $1$ and one double root at $1$) that the surface $\mathscr{Q}$ 
has genus zero, and is thus homeomorphic to a sphere $\mathbb{C}\cup \{\infty\}$,
see \textit{e.g.}\ Parts 4.17 and 5.12 of \cite{JS}. Therefore, this Riemann surface can be rationally 
uniformized, in the sense that it is possible to find two rational functions $x(z)$ 
and $y(z)$, such that $\mathscr{Q}=\{(x(z),y(z)) : z\in \mathbb{C}\cup \{\infty\} \}$~;
moreover, a standard uniformization (for an account of the concept of uniformization, 
see Part 4.9 of \cite{JS}) is~:
     \begin{equation}
     \label{uniformization}
          x\left(z\right) = \frac{(z-z_{1})(z-1/z_{1})}
          {(z-z_{0})(z-1/z_{0})},\ \ \ \ \ 
          y\left(z\right) = \frac{(z-K z_{3})(z-K/z_{3})}
          {(z-K z_{2})(z-K/z_{2})},
     \end{equation}
where 
     \begin{eqnarray*}
          z_{0}&=&\big[2-(x_{1}+x_{4})+2[(1-x_{1})(1-x_{4})]^{1/2}\big]\big/\big[x_{4}-x_{1}\big],\\
          z_{1}&=&\big[x_{1}+x_{4}-2x_{1}x_{4}+2[x_{1}x_{4}(1-x_{1})(1-x_{4})]^{1/2}\big]\big/\big[x_{4}-x_{1}\big],\\
          z_{2}&=&\big[2-(y_{1}+y_{4})+2[(1-y_{1})(1-y_{4})]^{1/2}\big]\big/\big[y_{4}-y_{1}\big],\\
          z_{3}&=&\big[y_{1}+y_{4}-2y_{1}y_{4}+2[y_{1}y_{4}(1-y_{1})(1-y_{4})]^{1/2}\big]\big/\big[y_{4}-y_{1}\big],
     \end{eqnarray*}
and where $K$ is a complex number of modulus $1$. Note that $z_{0}$ and $z_{1}$ 
(resp.\ $z_{2}$ and $z_{3}$)~have a modulus equal to one or are real, according 
to the signs of $x_{1}$ and $x_{4}$ (resp.\ $y_{1}$ and $y_{4}$).

For example, in the case of ${\rm SU}(3)$, it follows from a direct calculation
that 
     \begin{equation*}
          z_{0}=\exp(-2\imath \pi/3),\ \ z_{1}=1,\ \ z_{2}=\exp(-\imath \pi/3), 
          \ \ z_{3}=\exp(\imath \pi/3),\ \ K=\exp(-\imath\pi/3).
     \end{equation*} 
Above and throughout the paper, we note $\imath$ the usual complex number 
verifying~$\imath^{2}=-1$.

In the general case, in order to find $K$, we need to introduce a group of automorphisms 
naturally associated with the surface $\mathscr{Q}$. To begin with, let us remark that, with 
the previous notations, $Q(x,y)=0$ entails $Q(x,[c(x)/a(x)]/y)=0$ and $Q([\tilde{c}(y)/
\tilde{a}(y)]/x,y)=0$~; it is therefore natural to consider the group generated by the 
two bilinear transfor-mations $\hat{\xi}(x,y)=(x,[c(x)/a(x)]/y)$ and $\hat{\eta}(x,y)=
([\tilde{c}(y)/\tilde{a}(y)]/x,y)$, which is called, in \cite{FIM}, the \emph{group of 
the random walk}.

These automorphisms $\hat{\xi}$ and $\hat{\eta}$ define two automorphisms $\xi$ and 
$\eta$ of the uniformization space $\mathbb{C}\cup\{\infty\}$, characterized by~:
     \begin{equation}
     \label{characterization_automorphisms}
          \xi\circ \xi = 1,\ \ 
          x\circ \xi = x,\ \ 
          y\circ \xi = [c(x)/a(x)]/y,\ \ 
          \eta\circ \eta = 1,\ \ 
          y\circ \eta = y,\ \ 
          x\circ \eta=[\widetilde{c}(y)/\widetilde{a}(y)]/x.
     \end{equation} 
With~(\ref{uniformization}) and~(\ref{characterization_automorphisms}),
we obtain that they are equal to~:
     \begin{equation}
     \label{def_automorphisms_C_z}
          \xi(z)  = 1/z,\ \ \ \ \ 
          \eta(z) = K^{2}/z.
     \end{equation}
In particular, it is immediate that the group $W=\langle\xi ,\eta \rangle$
generated by $\xi$ and $\eta$ is isomorphic to the dihedral group of 
order $2\inf\{n> 0 : K^{2n}=1\}$. For example, in the case of SU($3$)
for which $K=\exp(-\imath \pi/3)$, $W$ is of order six -- this fact is
(differently) proved in Part~4.1 of~\cite{FIM}.

A crucial fact is that this property is actually verified by \emph{any} random
walk of $\mathscr{P}_{\alpha,\beta}$, since we have the following.
  \begin{prop}
  \label{Prop_value_K}
       For any process of $\mathscr{P}_{\alpha,\beta}$, $K=\exp(-\imath\pi/3)$.
  \end{prop}
\begin{proof}
With~(\ref{uniformization}), we have $y(K)=y_{1}$~; in addition,
by~(\ref{def_automorphisms_C_z}), $\eta(K)=K$, so that with
(\ref{characterization_automorphisms}), we obtain $x(K)^{2}
=\tilde{c}(y_{1})/\tilde{a}(y_{1})$. This implies that $x(K)=-
[\tilde{c}(y_{1})/\tilde{a}(y_{1})]^{1/2}$ -- indeed, we easily 
show that the roots of $Q(x,y_{1})$ have to be negative. By using 
again~(\ref{uniformization}), we get $K+1/K=\big[\tilde{a}(y_{1}
)^{1/2}(z_{1}+1/z_{1})+\tilde{c}(y_{1})^{1/2}(z_{0}+1/z_{0})\big]
\big/\big[\tilde{a}(y_{1})^{1/2}+\tilde{c}(y_{1})^{1/2}\big]$. In 
particular, $K+1/K$ can be expressed explicitly in terms of the 
jump probabilities $(p_{i,j})_{i,j}$. By using then Remark
\ref{description_walks} and after simplification, we get 
$K=\exp(-\imath\pi/3)$.
\end{proof}

\emph{From now on, we suppose that the process belongs to $\mathscr{P}_{\alpha,\beta}$. }
\medskip

For a better understanding of the surface $\mathscr{Q}$ as well as for
a coming use, we are now going to be interested in the transformations 
through the uniformization $(x,y)$ of some important cycles, namely 
the branch cuts $[x_{1},x_{4}]$, $[y_{1},y_{4}]$ and the unit circles 
$\{|x|=1\}$, $\{|y|=1\}$. First, by using~(\ref{uniformization}) and 
Proposition~\ref{Prop_value_K}, we immediately obtain~:
     \begin{equation}
     \label{eq_cycles}
          x^{-1}([x_{1},x_{4}])=\mathbb{R}\cup\{\infty\},\ \ \
          y^{-1}([y_{1},y_{4}])=\exp(-\imath\pi/3) \mathbb{R}\cup\{\infty\}.
     \end{equation}

As for the cycles $x^{-1}(\{|x|=1\})$ and $y^{-1}(\{|y|=1\})$,
their explicit expression~(calculated starting from~(\ref{uniformization}))
shows that they are real elliptic curves, which are located as in the middle of
Picture~\ref{transformation_cycles} below.

 \begin{figure}[!ht]
 \begin{picture}(0.00,130.00)
 \includegraphics{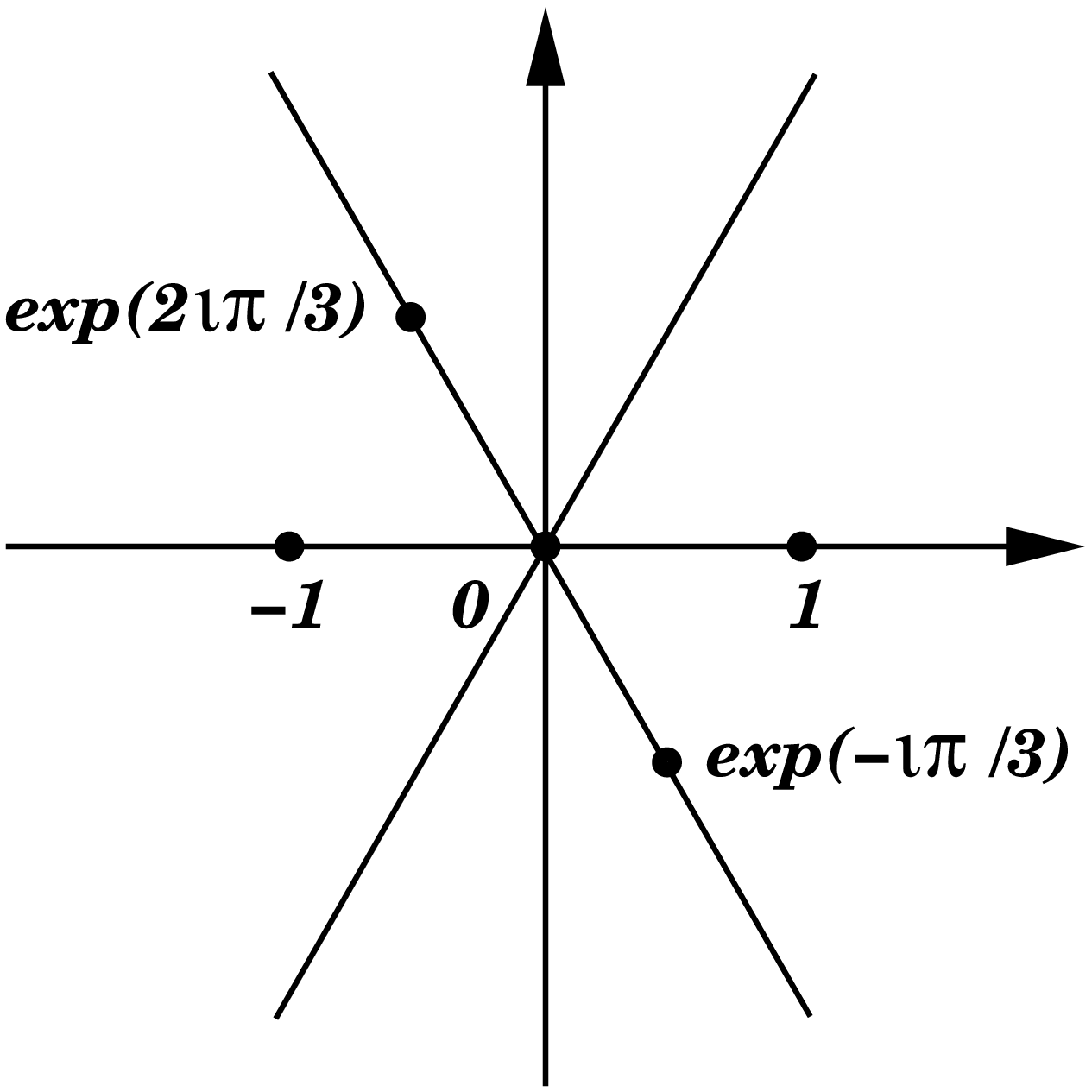}
 \hspace{49.5mm}
 \includegraphics{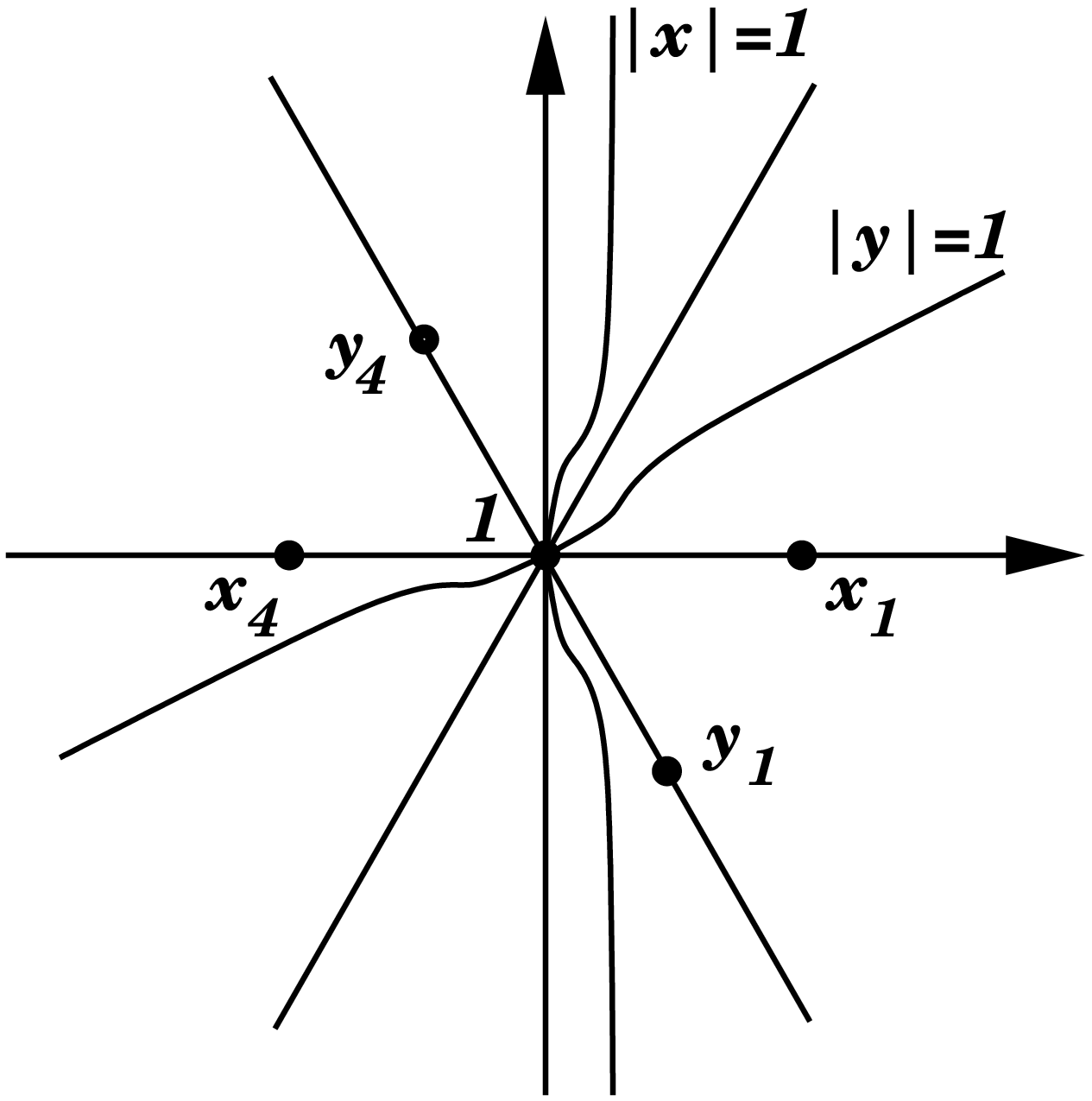}
 \hspace{49.5mm}
 \includegraphics{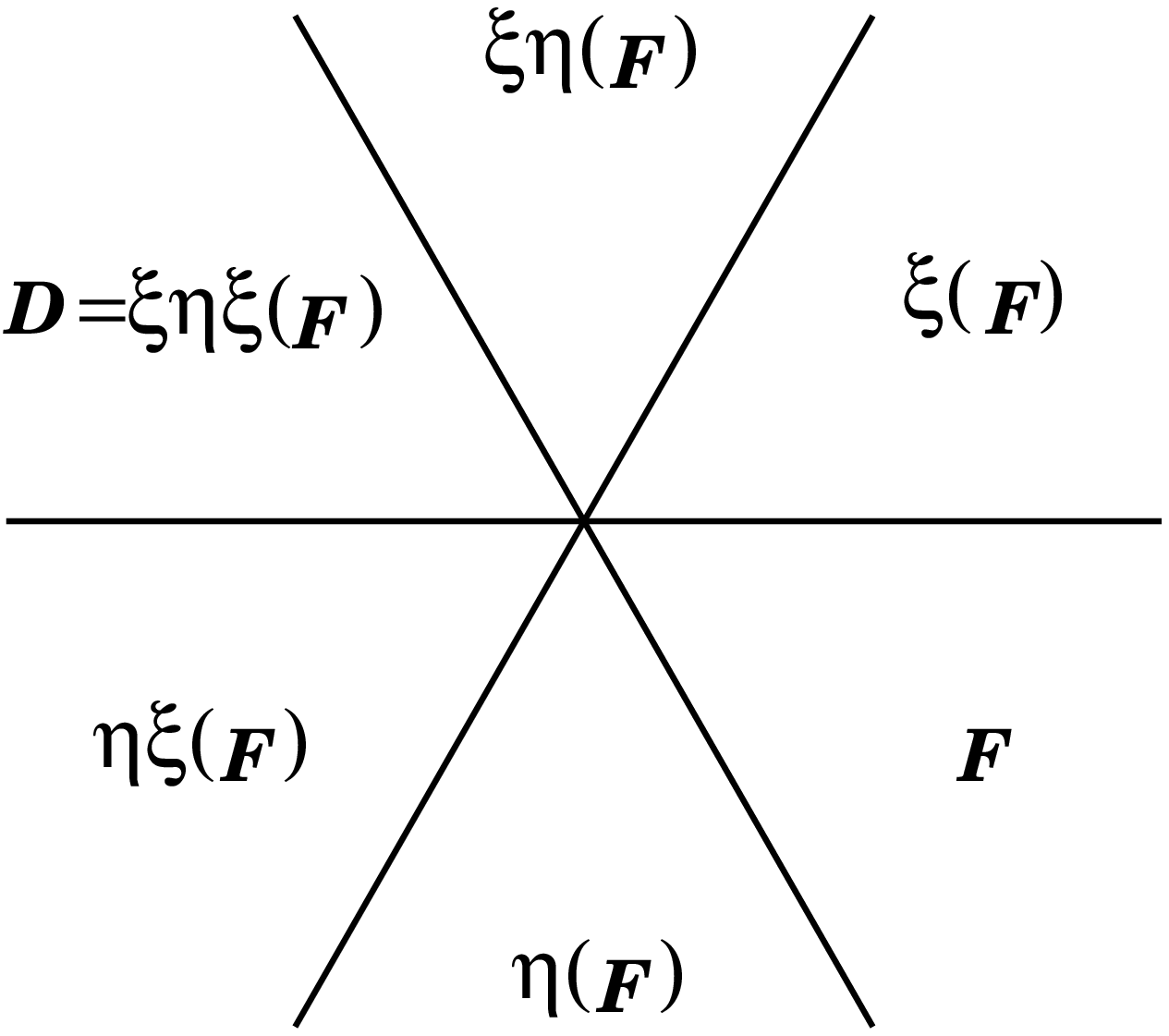}
 \end{picture}
 \caption{The uniformization space $\mathbb{C}\cup\{\infty\}$, with on the left some
          important elements of it, in the middle the corresponding elements
          through the uniformization $(x,y)$, and on the right the images of 
          the cone $F=\{x \exp(\imath \theta) : x\geq 0, -\pi/3\leq \theta\leq 0\}$ 
          through the six elements of the group $W=\langle \xi,\eta \rangle$}
 \label{transformation_cycles}
 \end{figure}

Note also that with~(\ref{def_automorphisms_C_z}) and Proposition~\ref{Prop_value_K},
we immediately obtain $\xi(\exp(\imath \theta)\mathbb{R}_{+})=\exp(-\imath \theta)
\mathbb{R}_{+}$ and $\eta(\exp(\imath \theta)\mathbb{R}_{+})=\exp(-\imath 
(\theta+2\pi/3))\mathbb{R}_{+}$. In particular, if we denote by $F$ the set $\{x \exp(\imath 
\theta) : x\geq 0, -\pi/3\leq \theta\leq 0\}$, we have --~see also on the right of 
Picture~\ref{transformation_cycles}~--
     \begin{equation}
     \label{fundamental_domain_first_step}
          \bigcup_{w\in W} w\left(F\right)=\mathbb{C}.
     \end{equation} 
 

Thanks to the group $W=\langle \xi,\eta\rangle$ and to
(\ref{fundamental_domain_first_step}), we are now going 
to continue the lifted functions $H^{i_{0},j_{0}}(z)=
h^{i_{0},j_{0}}(x(z))$ and $\tilde{H}^{i_{0},j_{0}}(z)
=\tilde{h}^{i_{0},j_{0}}(y(z))$~; this fact will turn 
out to be of the highest importance in the proof of 
Theorem~\ref{exp_Green_functions} -- the latter being 
crucial, since it will be the starting point of the forthcoming
Section~\ref{Martin_boundary}.

Note that in the sequel, we are often going to write $x^{i_{0}}y^{j_{0}}(z)$ instead of 
$x(z)^{i_{0}}y(z)^{j_{0}}$.

\begin{prop}
\label{continuation_h_h_tilde_covering}
The functions $H^{i_{0},j_{0}}(z)=h^{i_{0},j_{0}}(x(z))$ and 
$\tilde{H}^{i_{0},j_{0}}(z)=\tilde{h}^{i_{0},j_{0}}(y(z))$ 
can be meromorphically continued from respectively
$\{z\in \mathbb{C} : |x(z)|\leq 1\}$ and 
$\{z\in \mathbb{C} : |y(z)|\leq 1\}$ up to
respectively $\mathbb{C}\setminus \exp(\imath \pi)[0,\infty]$ and
$\mathbb{C}\setminus \exp(2\imath \pi/3)[0,\infty]$. 
These continuations verify
     \begin{equation}
     \label{stab_relations}
          H^{i_{0},j_{0}}\left(z\right)=
          H^{i_{0},j_{0}}\left(\xi\left(z\right)\right),\ \ \ \ \ 
          \widetilde{H}^{i_{0},j_{0}}\left(z\right)=
          \widetilde{H}^{i_{0},j_{0}}\left(\eta\left(z\right)\right)
     \end{equation}
for all $z\in \mathbb{C}$, and
     \begin{equation*}
          H^{i_{0},j_{0}}\left(z\right)+ \widetilde{H}^{i_{0},j_{0}}\left(z\right)
          +h_{0,0}^{i_{0},j_{0}}-x^{i_{0}}y^{j_{0}}\left(z\right)=\hspace{-0.5mm}
          \left\{\phantom{\begin{array}{ccccccc}
          \displaystyle 0 & \text{if} & z\in \mathbb{C}\setminus F_{4} \\
          \displaystyle  -\sum_{w\in W_{n}}\left(-1\right)^{l\left(w\right)}
          x^{i_{0}}y^{j_{0}}\left(w(z)\right) & \text{if} & z\in F_{4}.
          \end{array}}\right.\phantom{ii}\phantom{iii}
     \end{equation*}

\vspace{-21.5mm}

     \begin{subequations}
     \begin{eqnarray}
          \phantom{H^{i_{0},j_{0}}\left(z\right)+\widetilde{H}^{i_{0},j_{0}}\left(z\right)+
          x^{i_{0}}y^{j_{0}}\left(z\right) }&&\ \hspace{25mm} 0 \ \ 
          \hspace{25mm} \textit{if}\ \  z\in \mathbb{C}\setminus D
          \phantom{iiiiii} \label{eq_s_s_functions_1}\\
          \phantom{H\left(z\right)+\widetilde{H}\left(z\right)-
          x^{i_{0}}y^{j_{0}}\left(z\right)}&&\ \ \ \ \ -\sum_{w\in W}\left(-1\right)^{l\left(w\right)}
          x^{i_{0}}y^{j_{0}}\left(w(z)\right)\   \textit{if}\ \ z\in D
          \phantom{iiiiii} \label{eq_s_s_functions_2}
          \end{eqnarray}
\end{subequations}
where we have set $D=\{x\exp(\imath \theta) : x\geq 0, 2\pi/3\leq \theta\leq \pi\}$
and $l(w)$ for the length of $w$, \textit{i.e.}\ the smallest $r$ for which we can 
write $w=w_{1} \cdots  w_{r}$, with $w_{1},\ldots ,w_{r}$ equal to $\xi$ or $\eta$.
\end{prop}

\begin{rem}
In $\{z\in\mathbb{C} : |x(z)|\leq 1,|y(z)|\leq 1\}\subset \mathbb{C}\setminus D$,
(\ref{eq_s_s_functions_1}) follows directly~from~(\ref{functional_equation}).
\end{rem}


\begin{proof}[Proof of Proposition~\ref{continuation_h_h_tilde_covering}] In order 
to prove Proposition~\ref{continuation_h_h_tilde_covering}, we are going to use 
strongly the decomposition (\ref{fundamental_domain_first_step})~: precisely, we 
are going to define $H^{i_{0},j_{0}}$ and $\tilde{H}^{i_{0},j_{0}}$ piecewise, by 
defining them on each of the six domains $w(F)$ that appear in the decomposition
(\ref{fundamental_domain_first_step}), to be equal to some functions $H^{i_{0},j_{0}}_{w}$ 
and $\tilde{H}^{i_{0},j_{0}}_{w}$. It will then be enough to show that the functions 
$H^{i_{0},j_{0}}$ and $\tilde{H}^{i_{0},j_{0}}$ so defined verify the conclusions 
of Proposition~\ref{continuation_h_h_tilde_covering}.

\medskip

$\bullet$ In $F=\{x \exp(\imath \theta) : x\geq 0, -\pi/3\leq \theta\leq 0\}
\subset \{z \in \mathbb{C} : |x(z)|\leq 1,|y(z)|\leq 1\}$, see Picture
\ref{Lie}, we are going to use the most natural way to define $H^{i_{0},j_{0}}$ 
and $\tilde{H}^{i_{0},j_{0}}$, \textit{i.e.}\ their power series. So we set, 
for $z\in F$, $H_{1}^{i_{0},j_{0}}(z)=h^{i_{0},j_{0}}(x(z))$ and 
$\tilde{H}_{1}^{i_{0},j_{0}}(z)=\tilde{h}^{i_{0},j_{0}}(y(z))$
--~the subscript $1$ standing for the identity element of the group 
$W=\langle \xi,\eta \rangle$.

\medskip

$\bullet$ Next, we define $H_{\xi}^{i_{0},j_{0}}$, $\tilde{H}_{\xi}^{i_{0},j_{0}}$ on 
$\xi(F)$ and $H_{\eta}^{i_{0},j_{0}}$, $\tilde{H}_{\eta}^{i_{0},j_{0}}$ on $\eta(F)$ by
     \begin{equation*}
          \left.\begin{array}{ccccccc}
          \displaystyle\forall z\in \xi\left(F\right) & : &
          \displaystyle H_{\xi}^{i_{0},j_{0}}\big(z\big)=H_{1}^{i_{0},j_{0}}\big(\xi(z)\big),
          & \displaystyle \widetilde{H}_{\xi}^{i_{0},j_{0}}\big(z\big)=
          -H_{\xi}^{i_{0},j_{0}}\big(z\big)-h_{0,0}^{i_{0},j_{0}}+x^{i_{0}}y^{j_{0}}\big(z\big),\\
          \displaystyle \forall z\in \eta\left(F\right) & : &
          \displaystyle \widetilde{H}_{\eta}^{i_{0},j_{0}}\big(z\big)=
          \widetilde{H}_{1}^{i_{0},j_{0}}\big(\eta(z)\big), & \displaystyle
          H_{\eta}^{i_{0},j_{0}}\big(z\big)=-\widetilde{H}_{\eta}^{i_{0},j_{0}}\big(z\big)
          -h_{0,0}^{i_{0},j_{0}}+x^{i_{0}}y^{j_{0}}\big(z\big).
          \end{array}\right.
     \end{equation*}

$\bullet$ Then, we define $H^{i_{0},j_{0}}_{\xi\eta}$, 
$\tilde{H}^{i_{0},j_{0}}_{\xi\eta}$ on 
$\xi\eta(F)$ and $H_{\eta\xi}^{i_{0},j_{0}}$, $\tilde{H}_{\eta\xi}^{i_{0},j_{0}}$ on $\eta\xi(F)$ by
     \begin{equation*}
          \left.\begin{array}{ccccccc}
          \displaystyle \forall z\in \xi\eta\left(F\right) & : &
          \displaystyle H_{\xi\eta}^{i_{0},j_{0}}\big(z\big)=H_{\eta}^{i_{0},j_{0}}\big(\xi(z)\big),
          & \displaystyle \widetilde{H}_{\xi\eta}^{i_{0},j_{0}}\big(z\big)=
          -H_{\xi\eta}^{i_{0},j_{0}}\big(z\big)-h_{0,0}^{i_{0},j_{0}}+x^{i_{0}}y^{j_{0}}\big(z\big),\\
          \displaystyle \forall z\in \eta\xi\left(F\right) & : & \displaystyle
          \widetilde{H}_{\eta\xi}^{i_{0},j_{0}}\big(z\big)=
          \widetilde{H}_{\xi}^{i_{0},j_{0}}\big(\eta(z)\big), &  \displaystyle
          H_{\eta\xi}^{i_{0},j_{0}}\big(z\big)=-\widetilde{H}_{\eta\xi}^{i_{0},j_{0}}\big(z\big)
          -h_{0,0}^{i_{0},j_{0}}+x^{i_{0}}y^{j_{0}}\big(z\big).
          \end{array}\right.
     \end{equation*}

$\bullet$ At last, we define $H^{i_{0},j_{0}}_{\xi\eta\xi}$ and 
$\tilde{H}^{i_{0},j_{0}}_{\xi\eta\xi}$ on $\xi\eta\xi(F)=\eta\xi\eta(F)$ by
     \begin{equation*}
          \forall z\in \xi\eta\xi\left(F\right) \ \ : \ \ 
          H^{i_{0},j_{0}}_{\xi\eta\xi}\big(z\big)=H^{i_{0},j_{0}}_{\eta\xi}
          \big(\xi(z)\big),\ \ \ 
          \tilde{H}^{i_{0},j_{0}}_{\xi\eta\xi}\big(z\big)=\tilde{H}^{i_{0},j_{0}}_{\xi\eta}
          \big(\eta(z)\big).
     \end{equation*}

Therefore we have, for each of the six domains $w(F)$ of
the decomposition~(\ref{fundamental_domain_first_step}),
defined two functions $H_{w}^{i_{0},j_{0}}$ and $\tilde{H}_{w}^{i_{0},j_{0}}$. Then,
as said at the beginning of the proof, we set
$H^{i_{0},j_{0}}(z)=H_{w}^{i_{0},j_{0}}(z)$ and $\tilde{H}^{i_{0},j_{0}}(z)=\tilde{H}_{w}^{i_{0},j_{0}}(z)$
for all $z\in w(F)$ and all $w\in W$.

With this construction,~(\ref{stab_relations}) and~(\ref{eq_s_s_functions_1})
are immediately obtained. To prove~(\ref{eq_s_s_functions_2}), we can use
the fact that it is possible to express \emph{all} the functions $H_{w}^{i_{0},j_{0}}$, 
$\tilde{H}_{w}^{i_{0},j_{0}}$ in terms \emph{only} of $H_{1}^{i_{0},j_{0}},\tilde{H}_{1}^{i_{0},j_{0}},
h_{0,0}^{i_{0},j_{0}},x^{i_{0}}y^{j_{0}}$~: we give \textit{e.g.}\ the expressions of
$H_{\xi\eta\xi}^{i_{0},j_{0}}$ and $\tilde{H}_{\xi\eta\xi}^{i_{0},j_{0}}$ on $\xi\eta\xi(F)$~:
     \begin{eqnarray*}
          H_{\xi\eta\xi}^{i_{0},j_{0}}\big(z\big)={H}_{1}^{i_{0},j_{0}}
          \big(\xi\eta\xi(z)\big)-
          x^{i_{0}}y^{i_{0}}\big(\eta\xi(z)\big)+
          x^{i_{0}}y^{i_{0}}\big(\xi(z)\big),\\
          \widetilde{H}^{i_{0},j_{0}}_{\xi\eta\xi}\big(z\big)=\widetilde{H}_{1}^{i_{0},j_{0}}
          \big(\xi\eta\xi(z)\big)-
          x^{i_{0}}y^{i_{0}}\big(\xi\eta(z)\big)+
          x^{i_{0}}y^{i_{0}}\big(\eta(z)\big).
     \end{eqnarray*}
We therefore obtain~(\ref{eq_s_s_functions_2}), since with~(\ref{eq_s_s_functions_1}) we get
$H_{1}^{i_{0},j_{0}}\big(\xi\eta\xi(z)\big)+\tilde{H}_{1}^{i_{0},j_{0}}\big(\xi\eta\xi(z)\big)
+h_{0,0}^{i_{0},j_{0}}=x^{i_{0}}y^{j_{0}}\big(\xi\eta\xi(z)\big)$ for $z\in\xi\eta\xi(F)$, and
since $W=\{1,\xi,\eta,\eta\xi,\xi\eta,\xi\eta\xi\}$.
\end{proof}

\begin{thm}
\label{exp_Green_functions}
For any $i,j,i_{0},j_{0}>0$, 
     \begin{equation}
     \label{final_formulation_G_i_j}
          G_{i,j}^{i_{0},j_{0}}=\frac{-[z_{0}-1/z_{0}]/\Omega_{x}}
          {2\pi\imath[{p_{1,0}}^{2}-4p_{1,1}p_{1,-1}]^{1/2}}
          \int_{\exp\left(\imath\theta \right)[0,\infty]}
          \left[\frac{1}{z}\sum_{w\in W}\left(-1\right)^{l(w)}
          x^{i_{0}}y^{j_{0}}\left(w\left(z\right)\right)\right]
          \frac{\textnormal{d}z}{x\left(z\right)^{i}
          y\left(z\right)^{j}},
     \end{equation}
where $\theta\in[2\pi/3,\pi]$ and where we have set $\Omega_{x}=z_{0}+1/z_{0}-[z_{1}+1/z_{1}]
=4(x_{4}-1)(x_{1}-1)/(x_{4}-x_{1})<0$.
\end{thm}

\begin{proof}
We have already noticed that the generating function of the Green functions $G^{i_{0},j_{0}}$, 
defined in~(\ref{def_generating_functions}), is holomorphic in $\{(x,y)\in \mathbb{C}^{2} : 
|x|<1, |y|<1\}$. As a consequence and by using~(\ref{functional_equation}), the Cauchy formulas 
allow us to write its coefficients $G_{i,j}^{i_{0},j_{0}}$ as the following double integrals~:
     \begin{equation*}
     \label{application_Cauchy_formula}
          G_{i,j}^{i_{0},j_{0}}= 
          \frac{1}{\left[2\pi \imath \right]^{2}}
          \iint_{\substack{\{\left|x\right|=\left|y\right|=1\}}}
          \frac{h^{i_{0},j_{0}}\left(x\right)+\widetilde{h}^{i_{0},j_{0}}
          \left(y\right)+h_{0,0}^{i_{0},j_{0}}-x^{i_{0}}y^{j_{0}}}
          {Q\left(x,y\right)x^{i}y^{j}}\text{d}x\text{d}y.
     \end{equation*}
Then, by using the uniformization~(\ref{uniformization}), the location of the 
cycles $\{|x|=1\}$ and $\{|y|=1\}$, see Picture~\ref{transformation_cycles}, 
the residue theorem at infinity and the Cauchy theorem, we obtain that~:
     \begin{equation}
     \label{final_formulation_G_i_j_0}
          G_{i,j}^{i_{0},j_{0}}=\frac{1}{2\pi \imath}
          \int_{\exp\left(\imath\theta \right)[0,\infty]}
          \frac{H^{i_{0},j_{0}}\left(z\right)+\widetilde{H}^{i_{0},j_{0}}\left(z\right)
          +h_{0,0}^{i_{0},j_{0}}-x^{i_{0}}y^{j_{0}}\left(z\right)}
          {\big[\partial_{y}Q(x(z),y(z))\big]x(z)^{i}y(z)^{j}}
          x'\left(z\right)\text{d}z,
     \end{equation}
$\theta$ being any angle in $[2\pi/3,\pi]$ -- $[2\pi/3,\pi]$ because on the one hand, 
it is not possible to take $\theta>\pi$, since $\exp(\imath \pi)[0,\infty]$ is a 
singular curve for $H^{i_{0},j_{0}}$, and on the other hand, it is not allowed to 
have $\theta<2\pi/3$, since $\exp(2\imath \pi/3)[0,\infty]$ is a singular curve for 
$\tilde{H}^{i_{0},j_{0}}$, see Proposition~\ref{continuation_h_h_tilde_covering}.

Then, from~(\ref{uniformization}) and from the fact that
$\partial_{y}Q(x(z),y(z))=d(x(z))^{1/2}$, we remark that 
we have 
     \begin{equation*}
          x'(z)/\partial_{y}Q(x(z),y(z))=\frac{[z_{0}-1/z_{0}]/
          (z_{0}+1/z_{0}-[z_{1}+1/z_{1}])}{[{p_{1,0}}^{2}-4p_{1,1}p_{1,-1}]^{1/2}z}.
     \end{equation*}
In this way and by using~(\ref{eq_s_s_functions_2}) 
in~(\ref{final_formulation_G_i_j_0}), we get (\ref{final_formulation_G_i_j}).
\end{proof}

\section{Proof of Theorem~1~: asymptotic of the Green functions}
\label{Martin_boundary}

\paragraph{Beginning of the proof of Theorem~\ref{Main_theorem_Green_functions}.}

For any $\theta\in[2\pi/3,\pi]$, the function $x(z)^{i}y(z)^{j}$ is, 
on $\exp(\imath\theta)[0,\infty]$, larger than $1$ in modulus, see 
Picture~\ref{transformation_cycles}. Moreover, it goes to $1$ when 
(and only when) $z$ goes to $0$ or to $\infty$. This is why it seems 
natural to decompose the contour $\exp(\imath\theta)[0,\infty]$ into 
a part near $0$, an other near $\infty$ and the remaining part, and 
to think that the parts near $0$ and $\infty$ will lead to the 
asymptotic of $G_{i,j}^{i_{0},j_{0}}$, and that the remaining part 
will lead to a negligible contribution. In this way appears the 
question of finding the best possible contour in order to achieve 
this idea~; in other words, it is a matter of finding the value of 
$\theta$ for which the calculation of the asymptotic of~(\ref{final_formulation_G_i_j}) 
on $\exp(\imath\theta)[0,\infty]$ will be the easiest, among all 
the possibilities $\theta\in[2\pi/3,\pi]$.

\medskip

For this, we are going to consider with details the function $x(z)^{i}y(z)^{j}$,
or, equivalently, the function $\chi_{j/i}(z)=\ln(x(z))+(j/i)\ln(y(z))$.
Incidentally, this is why, from now on, we suppose that $j/i\in[0,M]$, for some
$M<\infty$. Indeed, the function $\chi_{j/i}$ is manifestly not appropriate
to the values $j/i$ going to $\infty$, for such $j/i$, we will consider later
the function $(i/j)\chi_{j/i}(z)=(i/j)\ln(x(z))+\ln(y(z))$. Nevertheless,
$M$ can be so large as wished, and, in what follows, we assume that some 
$M>0$ is fixed.

Now we set $\chi_{j/i}(z)=\sum_{p\geq 0}\nu_{p}(j/i)z^{p}$, this function is
\textit{a priori} well defined for $z$ in a neighborhood of $0$. Moreover, with
(\ref{uniformization}), we obtain that $\nu_{0}(j/i)=0$ and that for all $p\geq 1$,
     \begin{equation}
     \label{expansion_chi_0}
          p\nu_{p}(j/i)=\big(z_{0}^{p}+1/z_{0}^{p}-[z_{1}^{p}+1/z_{1}^{p}]\big)
          +(j/i)\big(z_{2}^{p}+1/z_{2}^{p}-[z_{3}^{p}+1/z_{3}^{p}]\big)/K^{p}.   
     \end{equation}

Likewise, we easily prove, by using~(\ref{uniformization}), that for $z$ 
near $\infty$, $\chi_{j/i}(z)=\sum_{p\geq 0}\overline{\nu_{p}}(j/i)z^{-p}$.

Consider now the steepest descent path associated with $\chi_{j/i}$, that is 
the function $z_{j/i}(t)$ defined by $\chi_{j/i}(z_{j/i}(t))=t$. By inverting 
the latter equality, we immediately obtain that the half-line $(1/\nu_{1}(j/i))
[0,\infty]$ is tangent at $0$ and at $\infty$ to the steepest descent path. 

Now we set, for the sake of briefness, $\rho_{j/i}=1/\nu_{1}(j/i)$.
With this notation, let us now answer the question asked above, that 
dealt with finding the value of $\theta$ for which the asymptotic of 
$G_{i,j}^{i_{0},j_{0}}$ will be the most easily calculated~: we are 
going to choose $\theta=\arg(\rho_{j/i})\in[2\pi/3,\pi]$, and the 
decomposition of the contour $\exp(\imath \theta)[0,\infty]$ will be 
     \begin{equation*}
          \exp(\imath \theta)[0,\infty]=
          \big(\rho_{j/i}/|\rho_{j/i}|\big)
          \big([0,\epsilon]\cup]\epsilon,1/\epsilon[
          \cup [1/\epsilon,\infty]\big).
     \end{equation*}

By using this decomposition in~(\ref{final_formulation_G_i_j}), 
we consider now that the Green functions are the sum of three 
terms, and we are going to study successively the contribution 
of each of these terms.

But first of all, we simplify the expression of $\rho_{j/i}$.
Setting $\Omega_{y}=z_{2}+1/z_{2}-[z_{3}+1/z_{3}]=4(y_{4}-1)
(y_{1}-1)/(y_{4}-y_{1})$ and using~(\ref{expansion_chi_0}), we 
immediately obtain that~$\nu_{1}(j/i)=\Omega_{x}+(j/i)
\Omega_{y}/K$. But it turns out that for all the walks of 
$\mathscr{P}_{\alpha,\beta}$, we have $\Omega_{y}=\alpha\Omega_{x}$ 
--~this follows from a direct calculation starting from the explicit 
expression of the branch points $x_{1},x_{4},y_{1},y_{4}$ in terms 
of the jump probabilities $(p_{i,j})_{i,j}$ and by using Remark
\ref{description_walks}. Therefore we have~:
     \begin{equation}
     \label{def_rho_j/i}
          \rho_{j/i}=\frac{1}{\nu_{1}\left(j/i\right)}=\frac{1}{\Omega_{x}}
          \frac{1}{1+\left(j/i\right)\alpha\exp\left(\imath\pi/3\right)}.
     \end{equation}

\paragraph{Contribution of the neighborhood of $0$.}
In order to evaluate the asymptotic of the integral~(\ref{final_formulation_G_i_j}) on 
the contour $(\rho_{j/i}/|\rho_{j/i}|)[0,\epsilon]$, we are going to use the expansion 
of the function $(1/z)\sum_{w\in W}(-1)^{l(w)}x^{i_{0}}y^{j_{0}}(w(z))$ at $0$ -- expansion
that we will obtain in Equation (\ref{expansion_s_s}) below. This is why we begin by studying 
the asymptotic of the following integral, for any non-negative integer~$k$~:
     \begin{equation}
     \label{integral_general_p}
          \int_{\left(\rho_{j/i}/|\rho_{j/i}|\right)[0,\epsilon]}
          \frac{z^{k}}{x\left(z\right)^{i}
          y\left(z\right)^{j}}\text{d}z.
     \end{equation}

By using the equality $1/[x(z)^{i}y(z)^{j}]=\exp(-i\chi_{j/i}(z))$ as well as the 
expansion~(\ref{expansion_chi_0}) of $\chi_{j/i}$ at $0$ and then making the 
change of variable $z=\rho_{j/i} t$, we obtain that~(\ref{integral_general_p}) is 
equal to
     \begin{equation}
     \label{integral_after_expansion}
          \rho_{j/i}^{k+1}\int_{0}^{\epsilon/|\rho_{j/i}|}
          t^{k}\exp\left(-i t\right)\exp\Big(-i \nu_{2}\left(j/i\right)
          (\rho_{j/i}t)^{2}\Big)
          \exp\bigg(-i\sum_{p\geq 3}\nu_{p}\left(j/i\right)
          (\rho_{j/i}t)^{p}\bigg)
          \text{d}t.
     \end{equation}

Now we set $m=\max\{|z_{0}|,1/|z_{0}|,|z_{1}|,1/|z_{1}|,|z_{2}|,1/|z_{2}|,|z_{3}|,1/|z_{3}|\}$. 
Then, with~(\ref{expansion_chi_0}), we get $|\nu_{p}(j/i)|\leq 4m^{p}(1+M)$. Thus, for all $t\in
[0,\epsilon/|\rho_{j/i}|]$, $|-i\sum_{p\geq 3}\nu_{p}(j/i)(\rho_{j/i}t)^{p}|\leq 4
(1+M)i (m\epsilon)^{3}/(1-m\epsilon)$. This is why $\exp(-i\sum_{p\geq 3}\nu_{p}(j/i)
(\rho_{j/i}t)^{p})=1+O(i\epsilon^{3})$, the $O$ being independent of $j/i\in[0,M]$ and of 
$t\in[0,\epsilon/|\rho_{j/i}|]$. The integral~(\ref{integral_after_expansion}) can thus be 
calculated as
     \begin{equation*}
          (\rho_{j/i}/i)^{k+1}
          \big[1+O(i\epsilon^{3})\big]
          \int_{0}^{i\epsilon/|\rho_{j/i}|}t^{k}\exp\left(-t\right)
          \big[1-\nu_{2}\left(j/i\right)\rho_{j/i}^{2}t^{2}/i+
          O(t^{4}/i^{2})\big]\text{d}t.
     \end{equation*}

Let us now choose $\epsilon$ such that $i\epsilon/|\rho_{j/i}|\to \infty$ and 
$O(i\epsilon^{3})=o(1/i)$, \textit{e.g.}\ $\epsilon=1/i^{3/4}$.
For this choice of $\epsilon$, we obtain that the integral~(\ref{integral_general_p}) 
is equal to
     \begin{equation}
     \label{conclusion_integral_general_p}
          (\rho_{j/i}/i)^{k+1}\big[1+o(1/i)\big]
          \big[k!-\nu_{2}\left(j/i\right)
          \rho_{j/i}^{2}\left(k+2\right)!/i+O(1/i^{2})\big],
     \end{equation}
where the $o$ and $O$ above are independent of $j/i\in[0,M]$.

\medskip

We are presently ready to find the asymptotic of
the integral~(\ref{final_formulation_G_i_j})
on the contour $(\rho_{j/i}/|\rho_{j/i}|)[0,\epsilon]$.
To begin with, we have the following expansion in the neighborhood of $0$ 
(directly obtained from~(\ref{uniformization}),~(\ref{def_automorphisms_C_z})
and Remark~\ref{description_walks})~:
     \begin{equation}
     \label{expansion_s_s}
          \sum_{w\in W}\left(-1\right)^{l\left(w\right)}
          x^{i_{0}}y^{j_{0}}\left(w(z)\right)=
          (\imath 3^{3/2}/2) \alpha \Omega_{x}^{3} i_{0}j_{0}\left(
          i_{0}+\alpha j_{0}+\beta\right)z^{3}+O\left(z^{6}\right).
     \end{equation}

Equation~(\ref{expansion_s_s}) implies then that the integral~(\ref{final_formulation_G_i_j})
on the contour $(\rho_{j/i}/|\rho_{j/i}|)[0,\epsilon]$ equals
     \begin{equation*}
          \frac{-[z_{0}-1/z_{0}]/\Omega_{x}}
          {2\pi\imath[{p_{1,0}}^{2}-4p_{1,1}p_{1,-1}]^{1/2}}
          \int_{\left(\rho_{j/i}/|\rho_{j/i}|\right)[0,\epsilon]}
          \frac{ (\imath 3^{3/2}/2) \alpha \Omega_{x}^{3} i_{0}j_{0}\left(
          i_{0}+\alpha j_{0}+\beta\right)z^{2}+O(z^{5})}{x\left(z\right)^{i}
          y\left(z\right)^{j}}\text{d}z.
     \end{equation*}

So, with~(\ref{integral_general_p}) and~(\ref{conclusion_integral_general_p})
applied for $k=2$ and $k=5$, we obtain that 
the integral~(\ref{final_formulation_G_i_j})
on the contour $(\rho_{j/i}/|\rho_{j/i}|)[0,\epsilon]$ is equal to
     \begin{equation}
     \label{conclusion_contribution_neighborhood_0}
          \frac{-[z_{0}-1/z_{0}]3^{3/2}\alpha\Omega_{x}^{2}}
          {4\pi[{p_{1,0}}^{2}-4p_{1,1}p_{1,-1}]^{1/2}}i_{0}j_{0}
          \left(i_{0}+\alpha j_{0}+\beta\right)(\rho_{j/i}/i)^{3}
          \big[2-24\nu_{2}\left(j/i\right)
          \rho_{j/i}^{2}/i+o(1/i)\big].
     \end{equation}

\paragraph{Contribution of the neighborhood of $\infty$.}

The part of the contour close to $\infty$,
$(\rho_{j/i}/|\rho_{j/i}|)[1/\epsilon,\infty]$,
is related to the part $(\rho_{j/i}/|\rho_{j/i}|)[0,\epsilon]$
by the transformation $z\mapsto 1/\overline{z}$. Moreover,
it is immediate from~(\ref{uniformization}) that for
$f=x$, $f=y$, or $f=\sum_{w\in W}(-1)^{l(w)}x^{i_{0}}y^{j_{0}}(w)$,
$f(1/\overline{z})=\overline{f(z)}.$
Therefore, the change of variable $z\mapsto 1/\overline{z}$
immediately entails that the contribution of the
integral~(\ref{final_formulation_G_i_j}) near $\infty$ is the
complex conjugate of its contribution near~$0$.

\paragraph{Contribution of the intermediate part.} We first recall from 
Proposition~\ref{continuation_h_h_tilde_covering} that $D$ denotes 
$\{x\exp(\imath \theta) : x\geq 0, 2\pi/3\leq \theta \leq \pi\}$,
and we define $A_{\epsilon}=\{z\in \mathbb{C} : \epsilon \leq 
|z|\leq 1/\epsilon\}$. Clearly (see Picture~\ref{transformation_cycles})
there exist $\eta_{x,\epsilon}>0$ and $\eta_{y,\epsilon}>0$ such that 
for all $z\in D\cap A_{\epsilon}$, $|x(z)|\geq 1+\eta_{x,\epsilon}$ and 
$|y(z)|\geq 1+\eta_{y,\epsilon}$. In fact, since $x'(0)=\Omega_{x}\neq 0$
and $y'(0)=\Omega_{y}/K\neq 0$, it is possible to take $\eta_{x,\epsilon}
\geq \eta \epsilon$ and $\eta_{y,\epsilon}\geq \eta \epsilon$, for some 
$\eta>0$ independent of $\epsilon$ small enough.

Let us now consider the function 
     \begin{equation*}
          s(z)=\frac{1}{x^{i_{0}}y^{j_{0}}(z)}
          \left[\sum_{w\in W}(-1)^{l(w)}
          x^{i_{0}}y^{j_{0}}(w(z))\right],
     \end{equation*} 
and let us show that $\sup_{z\in D}|s(z)|$ is finite. 
For this, it is sufficient to prove that $s$ has no pole in the
closed domain $D\cup\{\infty\}$. 

By~(\ref{uniformization}), the only zeros of the denominator of $s$ are at $z_{1},
1/z_{1},K z_{3},K/z_{3}$ which, as we easily check, belong to $-(D\cup \overline{D})$. 
Also, by~(\ref{uniformization}) and~(\ref{def_automorphisms_C_z}), the only poles of 
the numerator of $s$ are at $K^{2k}z_{0},K^{2k}/z_{0},K^{2k+1}z_{2},K^{2k+1}/z_{2}$, 
for $k\in\{0,1,2\}$. Next, we verify that both $z_{0}$ and $K z_{2}$ belong to $D$, so that 
among the twelve previous points, in fact only $z_{0}$ and $K z_{2}$ are in $D$. But in 
the definition of $s$, we took care of 
dividing by $x^{i_{0}}y^{j_{0}}$, so that $s$ is in fact holomorphic near these two
points. Moreover, $s$ is clearly holomorphic at $\infty$. Finally, 
we have proved that the meromorphic function $s$ has no pole in the closed domain 
$D\cup\{\infty\}$, hence $s$ is bounded in $D\cup\{\infty\}$, in other words 
$\sup_{z\in D}|s(z)|$ is finite.

In particular, the modulus of the contribution of the integral~(\ref{final_formulation_G_i_j})
on the intermediate part $(\rho_{j/i}/|\rho_{j/i}|)]\epsilon,1/\epsilon[
\subset D\cap A_{\epsilon}$ can be bounded from above by
     \begin{equation}
     \label{upper_bound_intermediate_term}
          \frac{|z_{0}-1/z_{0}|/|\Omega_{x}|}{2\pi |{p_{1,0}}^{2}
          -4p_{1,1}p_{1,-1}|^{1/2}} \frac{1}{\epsilon^{2}}
          \frac{\sup_{z\in D}\left|s\left(z\right)\right|}
          {(1+\eta\epsilon)^{i-i_{0}}
          (1+\eta\epsilon)^{j-j_{0}}}.
     \end{equation}
Note that the presence of the term $1/\epsilon^{2}$ in~(\ref{upper_bound_intermediate_term}) 
is due to the following~: one $1/\epsilon$ appears as an upper bound of the length of the 
contour, the other $1/\epsilon$ comes from an upper bound of the modulus of the term $1/z$ 
present in the integrand of~(\ref{final_formulation_G_i_j}).

Then, as before, we take $\epsilon=1/i^{3/4}$, and we use
the following straightforward upper bound, valid for $i$ large enough~: 
$1/(1+\eta/i^{3/4})^{i}\leq \exp(-(\eta/2) i^{1/4})$.
We finally obtain that for $i$ large 
enough,~(\ref{upper_bound_intermediate_term})
is equal to $O(i^{3/2}\exp(-(\eta/2) i^{1/4}))$.
We are going to see soon that this contribution is negligible w.r.t.\ the sum of the 
contributions of the integral~(\ref{final_formulation_G_i_j}) 
in the neighborhoods of $0$ and $\infty$.

\paragraph{Conclusion.}
We have shown that the contribution of the integral~(\ref{final_formulation_G_i_j})
in the~neighborhood of $0$ is given by~(\ref{conclusion_contribution_neighborhood_0}),
that the contribution of~(\ref{final_formulation_G_i_j}) in the neighborhood of 
$\infty$ is equal to the complex
conjugate of~(\ref{conclusion_contribution_neighborhood_0}), and that
the contribution of the remaining part equals
$O(i^{3/2}\exp(-(\eta/2) i^{1/4}))$. 
Moreover, starting from~(\ref{def_rho_j/i}), we immediately get that
$(\rho_{j/i}/i)^{3}-(\overline{\rho_{j/i}}/i)^{3}=
-\imath 3^{3/2}\alpha i j (i+\alpha j)/[\Omega_{x}(i^{2}+\alpha i j +\alpha^{2}j^{2})]^{3}$. 
In this way, we obtain
     \begin{eqnarray}
     \label{btc}
          G_{i,j}^{i_{0},j_{0}}=\frac{-[z_{0}-1/z_{0}]3^{3/2}\alpha\Omega_{x}^{2}}
          {4\pi({p_{1,0}}^{2}-4p_{1,1}p_{1,-1})^{1/2}}
          i_{0}j_{0}(i_{0}+\alpha j_{0}+\beta)\times \phantom{hhhhhhvcvdsfsdfsfdsfdfs}
          \\
          \phantom{fdsfsdfdfddf}\times \left[\frac{-2\imath 3^{3/2}\alpha i j (i+\alpha j)}
          {[\Omega_{x}(i^{2}+\alpha i j +\alpha^{2}j^{2})]^{3}}
          -24\frac{\nu_{2}(j/i){\rho_{j/i}}^{5}-
          \overline{\nu_{2}}(j/i)\overline{\rho_{j/i}}^{5}}{i^{4}}
          +o(1/i^{4})\right].\nonumber 
     \end{eqnarray}

If $\gamma\in ]0,\pi/2[$ and $j/i\to\tan(\gamma)$, then
$i j (i+\alpha j)/(i^{2}+\alpha i j +\alpha^{2}j^{2})^{3}\sim C_{\gamma,\alpha}/i^{3}$ with
$C_{\gamma,\alpha}>0$~: Theorem~\ref{Main_theorem_Green_functions}
for $\gamma\in]0,\pi/2[$ is thus an immediate consequence of~(\ref{btc}).
In that case, there was in fact no need to make an expansion with
two terms in~(\ref{conclusion_contribution_neighborhood_0})
and~(\ref{btc}) above, one single term would have been accurate enough. 

If $j/i\to \tan(0)=0$, then $i j (i+\alpha j)/(i^{2}+\alpha i j +
\alpha^{2}j^{2})^{3}\sim (j/i)/i^{3}$. By using the explicit expressions 
of $\nu_{2}(j/i)$ and $\rho_{j/i}$, see respectively~(\ref{expansion_chi_0})
and~(\ref{def_rho_j/i}), we easily obtain that $\nu_{2}(j/i){\rho_{j/i}}^{5}-
\overline{\nu_{2}}(j/i)\overline{\rho_{j/i}}^{5}=O(j/i)$. This implies that 
the sum of the two last terms in the brackets of~(\ref{btc}) equals $O((j/i)/i^{4})+
o(1/i^{4})$, which is obviously negligible w.r.t.\ $(j/i)/i^{3}$. 
Theorem~\ref{Main_theorem_Green_functions} is therefore also proved in the case $\gamma=0$.

In order to prove Theorem~\ref{Main_theorem_Green_functions} in the case 
$\gamma=\pi/2$, we would consider $(i/j)\kappa_{j/i}$ rather than $\kappa_{j/i}$,
and we would use then exactly the same analysis, we omit the details.

\medskip

Finally, in order to prove that the constant $C$ in the statement of
Theorem~\ref{Main_theorem_Green_functions} is positive, it is clearly
enough to show that $\imath [z_{0}-1/z_{0}]/[({p_{1,0}}^{2}-4p_{1,1}p_{1,-1})^{1/2} \Omega_{x}]$
is positive. 

For this, note first that from its definition, it is immediate that $\Omega_{x}<0$. Moreover, 
it follows from the beginning of Section~\ref{h_x_z} that if $x_{4}>0$, then $\imath 
[z_{0}-1/z_{0}]<0$ and $({p_{1,0}}^{2}-4p_{1,1}p_{1,-1})^{1/2}>0$~; if $x_{4}<0$, then 
$[z_{0}-1/z_{0}]>0$ and $\imath /({p_{1,0}}^{2}-4p_{1,1}p_{1,-1})^{1/2}<0$~; 
and if $x_{4}=\infty$,
by taking the limit in anyone of the two previous cases, we obtain
that $\imath [z_{0}-1/z_{0}]/[({p_{1,0}}^{2}-4p_{1,1}p_{1,-1})^{1/2}]<0$.

\paragraph{A few words about the analytical approach used here.}

  The two key steps in the proof of Theorem~\ref{Main_theorem_Green_functions} 
  are \emph{first} the explicit expression for the Green functions 
  (\ref{final_formulation_G_i_j_0}), and \emph{then} the expansion (\ref{eq_s_s_functions_2}) 
  of $H^{i_{0},j_{0}}+\tilde{H}^{i_{0},j_{0}}+h_{0,0}^{i_{0},j_{0}}-x^{i_{0}}y^{j_{0}}$ at $0$, 
  which is the numerator of the integrand in~(\ref{final_formulation_G_i_j_0}).

  It is worth noting that for any walk of $\mathscr{P}\supset \mathscr{P}_{\alpha,
  \beta}$, it is still possible to obtain~(\ref{final_formulation_G_i_j_0}) --~without 
  additional technical details, besides. On the other hand, obtaining explicitly the expansion 
  at $0$ of $H^{i_{0},j_{0}}+\tilde{H}^{i_{0},j_{0}}+h_{0,0}^{i_{0},j_{0}}-x^{i_{0}}
  y^{j_{0}}$ in the general setting seems us quite difficult -- all the more so as this 
  expansion has to comprise several terms, since \textit{a priori} it could happen 
  that several terms lead to non-negligible contributions in the asymptotic of the 
  Green functions. 

  It is more imaginable (though technically difficult) to obtain this expansion 
  for the walks for which
  an equality like~(\ref{eq_s_s_functions_2}) holds~; unfortunately, having such 
  an equality is far for being systematic, even for the processes associated with 
  a finite group $W$~: for example, the random walk with jump probabilities
  $p_{1,1}=p_{0,-1}=p_{-1,0}=1/3$ has manifestly a group $W$ of order six, but does not 
  verify an identity like~(\ref{eq_s_s_functions_2}).

\paragraph{Acknowledgments.}
I would like to thank Professor Bougerol for pointing out 
the interest and the relevance of the topic discussed in this paper.
I also warmly thank Professor Kurkova who introduced me to this 
field of research, as well as for her constant help and support 
during the elaboration of this article. Finally, I thank an
anonymous referee for his/her valuable comments, that have 
led me to substantially improve the first version of this paper.

\footnotesize

\bibliographystyle{alpha}
\bibliography{SU3f}

\end{document}